\documentclass[12pt]{article}
\usepackage{geometry,amsthm,amssymb,amsmath,enumerate,float,cite,algorithm2e,verbatim}
\newtheorem{theorem}{Theorem}

\newtheorem{lemma}{Lemma}

\title{Cubic graphs with equal independence number\\
and matching number}
\author{Elena Mohr \and Dieter Rautenbach}
\date{}

\begin{document}

\maketitle

{\small 
\begin{center}
Institute of Optimization and Operations Research, Ulm University, Germany\\
\texttt{\{elena.mohr,dieter.rautenbach\}@uni-ulm.de}
\end{center}
}

\begin{abstract}
Caro, Davila, and Pepper recently proved 
$\delta(G) \alpha(G)\leq \Delta(G) \mu(G)$
for every graph $G$ with 
minimum degree $\delta(G)$,
maximum degree $\Delta(G)$,
independence number $\alpha(G)$, and
matching number $\mu(G)$.
Answering some problems they posed, 
we characterize the extremal graphs 
for 
$\delta(G)<\Delta(G)$
as well as 
for $\delta(G)=\Delta(G)=3$.
\end{abstract}
{\small 
\begin{tabular}{lp{13cm}}
{\bf Keywords:} & Independence number; matching number\\
{\bf MSC 2010:} & 05C69, 
05C70 
\end{tabular}
}

\section{Introduction}

We consider finite, simple, and undirected graphs, and use standard terminology.
Recently, Caro, Davila, and Pepper \cite{cadape} 
proved the inequality
$$\delta(G) \alpha(G)\leq \Delta(G) \mu(G)$$
for every graph $G$ with 
minimum degree $\delta(G)$,
maximum degree $\Delta(G)$,
independence number $\alpha(G)$, and
matching number $\mu(G)$.
As an open problem they asked for 
the characterization of the extremal graphs,
that is, those graphs that satisfy this inequality with equality.
In particular, they asked for 
the characterization of the cubic graphs $G$
with $\alpha(G)=\mu(G)$.
In the present note, 
we give a simple proof of the above inequality,
which allows to characterize the non-regular extremal graphs.
Furthermore, we characterize the cubic extremal graphs.

\section{Results}

For positive integers $\delta$ and $\Delta$ with $\delta<\Delta$,
a bipartite graph is {\it $(\delta,\Delta)$-regular} if 
it has a bipartition with partite sets $A$ and $B$ 
such that every vertex in $A$ has degree $\delta$
and every vertex in $B$ has degree $\Delta$.
\begin{theorem}
If $G$ is a graph with minimum degree $\delta$
and maximum degree $\Delta$,
then 
\begin{eqnarray}\label{e1}
\delta \alpha(G)\leq \Delta \mu(G).
\end{eqnarray}
Furthermore, if $\delta<\Delta$, then
equality holds in (\ref{e1}) if and only if $G$ is 
bipartite and $(\delta,\Delta)$-regular.
\end{theorem}
\begin{proof}
Let $I$ be a maximum independent set in $G$.
Let $R=V(G)\setminus I$, 
and let $H$ be the bipartite spanning subgraph of $G$
that contains all edges of $G$ between $I$ and $R$.
Let $M$ be a maximum matching in $H$,
and let $U$ be a minimum vertex cover in $H$.
See Figure \ref{fig-1} for an illustration.

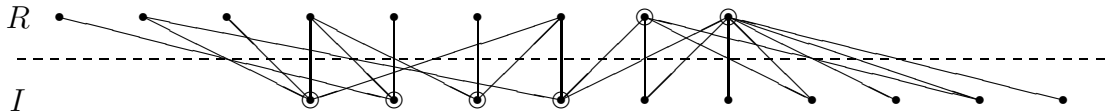
\begin{figure}[H]
$\mbox{}$\hfill
\unitlength 1.1mm 
\linethickness{0.4pt}
\ifx\plotpoint\undefined\newsavebox{\plotpoint}\fi 
\begin{picture}(130,17)(0,0)
\put(95,5){\circle*{1}}
\put(105,5){\circle*{1}}
\put(5,15){\circle*{1}}
\put(115,5){\circle*{1}}
\put(15,15){\circle*{1}}
\put(125,5){\circle*{1}}
\put(25,15){\circle*{1}}
\put(35,5){\circle*{1}}
\put(35,15){\circle*{1}}
\put(45,5){\circle*{1}}
\put(45,15){\circle*{1}}
\put(55,5){\circle*{1}}
\put(55,15){\circle*{1}}
\put(65,5){\circle*{1}}
\put(65,15){\circle*{1}}
\put(75,5){\circle*{1}}
\put(75,15){\circle*{1}}
\put(85,5){\circle*{1}}
\put(85,15){\circle*{1}}
\put(0,5){\makebox(0,0)[cc]{$I$}}
\put(0,15){\makebox(0,0)[cc]{$R$}}
\put(35,15){\line(0,-1){10}}
\put(45,15){\line(0,-1){10}}
\put(55,15){\line(0,-1){10}}
\put(65,15){\line(0,-1){10}}
\put(75,15){\line(0,-1){10}}
\put(85,15){\line(0,-1){10}}
\put(35,5){\circle{2}}
\put(45,5){\circle{2}}
\put(55,5){\circle{2}}
\put(65,5){\circle{2}}
\put(75,15){\circle{2}}
\put(85,15){\circle{2}}
\put(35,15){\line(1,-1){10}}
\put(35,15){\line(2,-1){20}}
\put(55,5){\line(1,1){10}}
\put(65,15){\line(-3,-1){30}}
\put(35,5){\line(-1,1){10}}
\put(5,15){\line(4,-1){40}}

\put(15,15){\line(5,-1){50}}
\put(75,15){\line(-1,-1){10}}
\put(65,5){\line(2,1){20}}
\put(85,15){\line(-1,-1){10}}
\put(15,15){\line(2,-1){20}}

\put(125,5){\line(-4,1){40}}
\put(85,15){\line(3,-1){30}}
\put(115,5){\line(-4,1){40}}
\put(75,15){\line(2,-1){20}}
\put(95,5){\line(-1,1){10}}
\put(85,15){\line(2,-1){20}}
\put(-.07,9.93){\line(1,0){.9924}}
\put(1.914,9.93){\line(1,0){.9924}}
\put(3.899,9.93){\line(1,0){.9924}}
\put(5.884,9.93){\line(1,0){.9924}}
\put(7.869,9.93){\line(1,0){.9924}}
\put(9.853,9.93){\line(1,0){.9924}}
\put(11.838,9.93){\line(1,0){.9924}}
\put(13.823,9.93){\line(1,0){.9924}}
\put(15.808,9.93){\line(1,0){.9924}}
\put(17.792,9.93){\line(1,0){.9924}}
\put(19.777,9.93){\line(1,0){.9924}}
\put(21.762,9.93){\line(1,0){.9924}}
\put(23.747,9.93){\line(1,0){.9924}}
\put(25.731,9.93){\line(1,0){.9924}}
\put(27.716,9.93){\line(1,0){.9924}}
\put(29.701,9.93){\line(1,0){.9924}}
\put(31.685,9.93){\line(1,0){.9924}}
\put(33.67,9.93){\line(1,0){.9924}}
\put(35.655,9.93){\line(1,0){.9924}}
\put(37.64,9.93){\line(1,0){.9924}}
\put(39.624,9.93){\line(1,0){.9924}}
\put(41.609,9.93){\line(1,0){.9924}}
\put(43.594,9.93){\line(1,0){.9924}}
\put(45.579,9.93){\line(1,0){.9924}}
\put(47.563,9.93){\line(1,0){.9924}}
\put(49.548,9.93){\line(1,0){.9924}}
\put(51.533,9.93){\line(1,0){.9924}}
\put(53.517,9.93){\line(1,0){.9924}}
\put(55.502,9.93){\line(1,0){.9924}}
\put(57.487,9.93){\line(1,0){.9924}}
\put(59.472,9.93){\line(1,0){.9924}}
\put(61.456,9.93){\line(1,0){.9924}}
\put(63.441,9.93){\line(1,0){.9924}}
\put(65.426,9.93){\line(1,0){.9924}}
\put(67.411,9.93){\line(1,0){.9924}}
\put(69.395,9.93){\line(1,0){.9924}}
\put(71.38,9.93){\line(1,0){.9924}}
\put(73.365,9.93){\line(1,0){.9924}}
\put(75.35,9.93){\line(1,0){.9924}}
\put(77.334,9.93){\line(1,0){.9924}}
\put(79.319,9.93){\line(1,0){.9924}}
\put(81.304,9.93){\line(1,0){.9924}}
\put(83.288,9.93){\line(1,0){.9924}}
\put(85.273,9.93){\line(1,0){.9924}}
\put(87.258,9.93){\line(1,0){.9924}}
\put(89.243,9.93){\line(1,0){.9924}}
\put(91.227,9.93){\line(1,0){.9924}}
\put(93.212,9.93){\line(1,0){.9924}}
\put(95.197,9.93){\line(1,0){.9924}}
\put(97.182,9.93){\line(1,0){.9924}}
\put(99.166,9.93){\line(1,0){.9924}}
\put(101.151,9.93){\line(1,0){.9924}}
\put(103.136,9.93){\line(1,0){.9924}}
\put(105.121,9.93){\line(1,0){.9924}}
\put(107.105,9.93){\line(1,0){.9924}}
\put(109.09,9.93){\line(1,0){.9924}}
\put(111.075,9.93){\line(1,0){.9924}}
\put(113.059,9.93){\line(1,0){.9924}}
\put(115.044,9.93){\line(1,0){.9924}}
\put(117.029,9.93){\line(1,0){.9924}}
\put(119.014,9.93){\line(1,0){.9924}}
\put(120.998,9.93){\line(1,0){.9924}}
\put(122.983,9.93){\line(1,0){.9924}}
\put(124.968,9.93){\line(1,0){.9924}}
\put(126.953,9.93){\line(1,0){.9924}}
\put(128.937,9.93){\line(1,0){.9924}}
\end{picture}
\hfill$\mbox{}$
\caption{The partition into $I$ and $R$.
The vertical edges form the matching $M$, 
and the encircled vertices form the vertex cover $U$.}\label{fig-1}
\end{figure}
Since $G$ is bipartite, 
K\"{o}nig's theorem \cite{ko} implies that
$\mu(H)=|M|=|U|$, and that $U$ intersects each edge in $M$ in exactly one vertex.
Let $m$ be the number of edges of $H$ between $I\setminus U$
and $U\cap R$.
Since each edge leaving $I\setminus U$ enters $U\cap R$,
we obtain, for $k=|I\cap U|$, that 
\begin{eqnarray}\label{e2}
\delta (\alpha(G)-k)=\delta |I\setminus U|\leq m\leq 
\Delta|U\cap R|=
\Delta(\mu(H)-k)\leq \Delta(\mu(G)-k).
\end{eqnarray}
This implies 
\begin{eqnarray}\label{e3}
\delta \alpha(G)\leq \delta \alpha(G)+(\Delta-\delta)k\leq \Delta \mu(H)\leq \Delta \mu(G),
\end{eqnarray}
that is, the inequality (\ref{e1}) follows.

We proceed to the characterization of the graphs $G$
that satisfy $\delta \alpha(G)=\Delta \mu(G)$
for $\delta<\Delta$.
Since $\alpha(G)$ and $\mu(G)$ are additive with respect to the components, it suffices to characterize the connected graphs.
If $G$ is bipartite and $(\delta,\Delta)$-regular,
and has partite sets $A$ and $B$ as above,
then Hall's theorem \cite{ha} and K\"{o}nig's theorem imply that 
$\mu(G)=|B|$, and that $\alpha(G)=n-\mu(G)=|A|$.
Furthermore, the number of edges of $G$ equals 
$\delta |A|=\delta \alpha(G)$
and 
$\Delta |B|=\Delta \mu(G)$,
that is, $\delta \alpha(G)=\Delta \mu(G)$.
Now, let $G$ be a connected graph
with $\delta \alpha(G)=\Delta \mu(G)$.
If $H$, $M$, $U$, $m$, and $k$ are as above,
then equality holds throughout (\ref{e2}) and (\ref{e3}).
This implies that $k=0$
and $\delta |I\setminus U|=m=\Delta|U\cap R|$,
which implies that 
$G[(I\setminus U)\cup (U\cap R)]$
is a non-empty bipartite graph, where 
every vertex in $I\setminus U$
has degree $\delta$
and 
every vertex in $U\cap R$
has degree $\Delta$.
Since all edges in $G$ leaving $I\setminus U$ enter $U\cap R$,
this implies that no edge of $G$ has only one endpoint in 
$(I\setminus U)\cup (U\cap R)$.
Since $G$ is connected, this implies 
that $G$ equals $G[(I\setminus U)\cup (U\cap R)]$,
that is bipartite and $(\delta,\Delta)$-regular.
This completes the proof. 
\end{proof}
Note that all cycles satisfy (\ref{e1}) with equality,
in particular, there are non-bipartite extremal graphs.
For higher degrees of regularity, that is, for $\delta=\Delta\geq 3$,
the extremal graphs have a richer structure,
which we elucidate for $\delta=\Delta=3$.
A graph $G$ is a {\it bubble with contact vertex $z$ and partition $(I,R)$} 
if the vertex set of $G$ can be partitioned into two sets $I$ and $R$ such that 
\begin{itemize}
\item every vertex in $V(G)\setminus \{ z\}$ has degree $3$ and $z$ has degree $2$,
\item $I$ is independent, and
\item $z$ lies in $R$ and $G[R]$ contains exactly one edge.
\end{itemize}
Since a bubble $G$ has degree sequence $3,\ldots,3,2$, it is not bipartite.
Counting the edges of $G$ implies that $|R|=|I|+1$,
and, hence, $|I|=(n(G)-1)/2$.
Figure \ref{fig1} illustrates some connected bubbles.

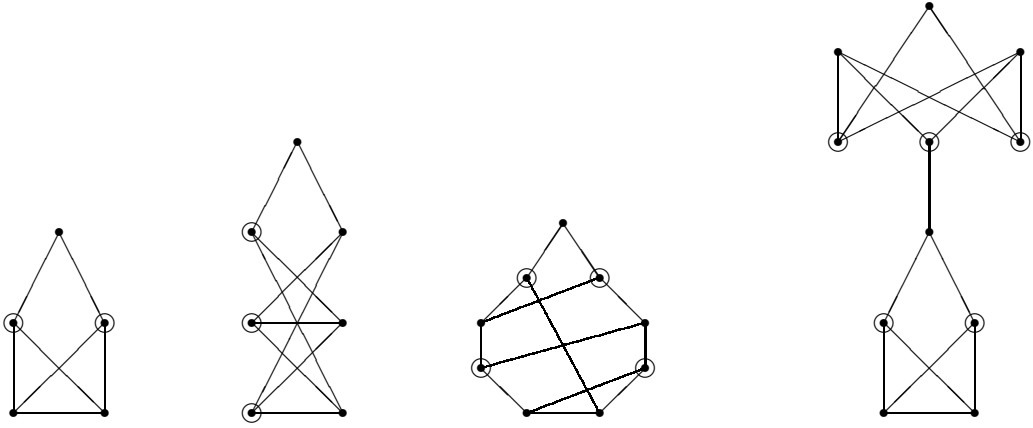
\begin{figure}[H]
$\mbox{}$\hfill
\unitlength 1.2mm 
\linethickness{0.4pt}
\ifx\plotpoint\undefined\newsavebox{\plotpoint}\fi 
\begin{picture}(17,26)(0,0)
\put(5,5){\circle*{1}}
\put(5,15){\circle*{1}}
\put(15,5){\circle*{1}}
\put(15,15){\circle*{1}}
\put(10,25){\circle*{1}}
\put(10,25){\line(-1,-2){5}}
\put(10,25){\line(1,-2){5}}
\put(15,15){\line(0,-1){10}}
\put(15,5){\line(-1,0){10}}
\put(5,5){\line(0,1){10}}
\put(5,15){\line(1,-1){10}}
\put(15,15){\line(-1,-1){10}}
\put(15,15){\circle{2}}
\put(5,15){\circle{2}}
\end{picture}
\hfill
\linethickness{0.4pt}
\ifx\plotpoint\undefined\newsavebox{\plotpoint}\fi 
\begin{picture}(16,36)(0,0)
\put(5,5){\circle*{1}}
\put(15,5){\circle*{1}}
\put(5,15){\circle*{1}}
\put(15,15){\circle*{1}}
\put(5,25){\circle*{1}}
\put(15,25){\circle*{1}}
\put(10,35){\circle*{1}}
\put(15,5){\line(-1,0){10}}
\put(5,5){\line(1,1){10}}
\put(15,15){\line(-1,0){10}}
\put(5,15){\line(1,-1){10}}
\put(15,5){\line(-1,2){10}}
\put(5,25){\line(1,2){5}}
\put(10,35){\line(1,-2){5}}
\put(15,25){\line(-1,-2){10}}
\put(5,15){\line(1,1){10}}
\put(5,25){\line(1,-1){10}}
\put(5,25){\circle{2}}
\put(5,15){\circle{2}}
\put(5,5){\circle{2}}
\end{picture}
\hfill
\linethickness{0.4pt}
\ifx\plotpoint\undefined\newsavebox{\plotpoint}\fi 
\begin{picture}(25,27)(0,0)
\put(5,10){\circle*{1}}
\put(23,10){\circle*{1}}
\put(10,5){\circle*{1}}
\put(10,20){\circle*{1}}
\put(18,5){\circle*{1}}
\put(18,20){\circle*{1}}
\put(5,15){\circle*{1}}
\put(23,15){\circle*{1}}
\put(14,26){\circle*{1}}
\put(18,20){\line(1,-1){5}}
\put(23,15){\line(0,-1){5}}
\put(23,10){\line(-1,-1){5}}
\put(18,5){\line(-1,0){8}}
\put(10,5){\line(-1,1){5}}
\put(5,10){\line(0,1){5}}
\put(5,15){\line(1,1){5}}
\put(10,20){\line(2,3){4}}
\put(14,26){\line(2,-3){4}}
\multiput(18,20)(-.087248322,-.033557047){149}{\line(-1,0){.087248322}}
\multiput(23,15)(-.120805369,-.033557047){149}{\line(-1,0){.120805369}}
\multiput(23,10)(-.087248322,-.033557047){149}{\line(-1,0){.087248322}}
\multiput(10,20)(.033613445,-.06302521){238}{\line(0,-1){.06302521}}
\put(18,20){\circle{2}}
\put(10,20){\circle{2}}
\put(23,10){\circle{2}}
\put(5,10){\circle{2}}
\end{picture}
\hfill
\linethickness{0.4pt}
\ifx\plotpoint\undefined\newsavebox{\plotpoint}\fi 
\begin{picture}(32,51)(0,0)
\put(15,5){\circle*{1}}
\put(15,15){\circle*{1}}
\put(25,5){\circle*{1}}
\put(25,15){\circle*{1}}
\put(20,25){\circle*{1}}
\put(20,25){\line(-1,-2){5}}
\put(20,25){\line(1,-2){5}}
\put(25,15){\line(0,-1){10}}
\put(25,5){\line(-1,0){10}}
\put(15,5){\line(0,1){10}}
\put(15,15){\line(1,-1){10}}
\put(25,15){\line(-1,-1){10}}
\put(25,15){\circle{2}}
\put(15,15){\circle{2}}
\put(20,35){\circle*{1}}
\put(30,35){\circle*{1}}
\put(30,45){\circle*{1}}
\put(10,35){\circle*{1}}
\put(10,45){\circle*{1}}
\put(20,35){\line(0,-1){10}}
\put(10,45){\line(0,-1){10}}
\put(20,35){\line(1,1){10}}
\put(30,45){\line(0,-1){10}}
\put(10,45){\line(1,-1){10}}
\put(10,45){\line(2,-1){20}}
\put(30,45){\line(-2,-1){20}}
\put(20,50){\circle*{1}}
\put(30,35){\line(-2,3){10}}
\put(20,50){\line(-2,-3){10}}
\put(10,35){\circle{2}}
\put(20,35){\circle{2}}
\put(30,35){\circle{2}}
\end{picture}
\hfill$\mbox{}$
\caption{Four connected bubbles; 
the contact vertices are the topmost vertices,
the encircled vertices form the sets $I$, and 
the rightmost bubble illustrates 
that a bubble may properly contain a smaller bubble.}\label{fig1}
\end{figure}

\begin{lemma}\label{lemma1}
If $G$ is a bubble with contact vertex $z$ and partition $(I,R)$,
then 
$$\alpha(G)=\alpha(G-z)=\mu(G)=\mu(G-z)=(n(G)-1)/2.$$
Furthermore, if $G$ is not $2$-connected, 
then some proper induced subgraph $G'$ of $G$ is also a bubble
with partition $(I',R')$ such that $I'\subseteq I$ and $R'\subseteq R$.
\end{lemma}
\begin{proof}
Let $p=(n(G)-1)/2$, and let $xy$ be the unique edge of $G[R]$.
Note that $z$ may coincide with $x$ or $y$.

Since every matching $M$ of $G$ 
either contains $xy$ and at most $|R|-2$ further edges
incident with the vertices in $R\setminus \{ x,y\}$
or does not contain $xy$ and at most $|I|$ edges
incident with the vertices in $I$,
we have $\mu(G)\leq |R|-1=|I|=p$.

Let $u$ be any vertex from $\{ x,y,z\}$.
For some set $S\subseteq I$, let $T=N_G(S)$.
Since $I$ is independent, 
we obtain $T\subseteq R$,
and the vertex degrees imply $|T|\geq |S|$.
Furthermore, if $T$ contains $u$,
then the edge $xy$ and the degree of $z$ imply
that $|T|\geq |S|+1$.
Altogether, $|N_{G-u}(S)|=|T\setminus \{ u\}|\geq |S|$
for every set $S\subseteq I$,
and Hall's Theorem implies the existence of a matching $M_u$ in $G-u$
that saturates each vertex in $I$, in particular, 
$\mu(G)=\mu(G-z)=p$.

Now, let $J$ be a maximum independent set in $G$.
If $J$ does not contain $x$, then $M_x$ and the edges $xy$ imply that $|J|\leq p$.
Similarly, if $J$ does not contain $y$, then $|J|\leq p$,
which implies $\alpha(G)\leq p$.
Since $I$ is an independent set of order $p$, 
we obtain that $\alpha(G)=\alpha(G-z)=p$.

Finally, suppose that $G$ is not $2$-connected.
If $G$ is not connected, 
then the vertex degrees easily imply
that the component of $G$ containing the edge $xy$ 
is also a bubble
with a partition $(I',R')$ such that $I'\subseteq I$ and $R'\subseteq R$.
Hence, we may assume that $G$ is connected but not $2$-connected.
Since $G$ is subcubic, this implies that $G$ has a bridge $uv$.
Let $G_u$ and $G_v$ be the components of $G-uv$ containing $u$ and $v$, respectively.
Let $I_u=V(G_u)\cap I$, 
$d_1=\sum_{w\in I_u}d_{G_u}(w)$, 
$R_u=V(G_u)\cap R$, and 
$d_2=\sum_{w\in R_u}d_{G_u}(w)$.
If $u,v\in R$, then $uv$ is the unique edge $xy$ in $G[R]$,
the graph $G_u$ is a bipartite graph with partite sets $I_u$ and $R_u$,
but $d_1$ and $d_2$ have different parities modulo $3$, which is a contradiction.
Hence, by symmetry, we may assume that $u\in R$ and $v\in I$.
Since $d_1$ is a multiple of $3$,
it follows that the unique edge $xy$ of $G[R]$ lies within $R_u$,
and that the contact vertex of $G$ does not lie $R_u$.
Hence, $G_u$ is a bubble with partition $(I_u,R_u)$, 
which completes the proof.
\end{proof}
A graph $G$ is {\it special} if it is connected, cubic, 
and the vertex set of $G$ can be partitioned into sets $V_0,V_1,\ldots,V_\ell$ such that
\begin{itemize}
\item the graph $G[V_0]$ is a non-empty bipartite graph with partite sets $I_0$ 
and $R_0$ 
such that every vertex in $R_0$ has degree $3$ in $G[V_0]$, and 
\item for every $i$ in $[\ell]$,
the graph $G[V_i]$ is a $2$-connected bubble with contact vertex $z_i$.
\end{itemize}
Note that, since $G$ is connected and $V_0$ is non-empty, 
it follows that $G[V_0]$ is connected, and that,
for every $i$ in $[\ell]$, 
the graph $G$ contains a bridge between $z_i$ 
and some vertex in $I_0$.
Since $G$ is cubic,
this implies $\ell=\sum_{u\in I_0}(3-d_{G[V_0]}(u))$.
In particular, if $\ell=0$, then $G$ is bipartite.
See Figure \ref{fig2} for an illustration.

\begin{figure}[H]
$\mbox{}$\hfill
\unitlength 1mm 
\linethickness{0.4pt}
\ifx\plotpoint\undefined\newsavebox{\plotpoint}\fi 
\begin{picture}(150,20)(0,0)
\put(10,5){\circle*{1}}
\put(45,5){\circle*{1}}
\put(80,15){\circle*{1}}
\put(20,5){\circle*{1}}
\put(55,5){\circle*{1}}
\put(90,15){\circle*{1}}
\put(5,15){\circle*{1}}
\put(40,15){\circle*{1}}
\put(75,5){\circle*{1}}
\put(120,5){\circle*{1}}
\put(110,15){\circle*{1}}
\put(15,15){\circle*{1}}
\put(50,15){\circle*{1}}
\put(85,5){\circle*{1}}
\put(130,5){\circle*{1}}
\put(120,15){\circle*{1}}
\put(25,15){\circle*{1}}
\put(60,15){\circle*{1}}
\put(95,5){\circle*{1}}
\put(140,5){\circle*{1}}
\put(140,15){\circle*{1}}
\put(130,15){\circle*{1}}
\put(5,15){\line(1,0){10}}
\put(40,15){\line(1,0){10}}
\put(15,15){\line(1,-2){5}}
\put(50,15){\line(1,-2){5}}
\put(20,5){\line(1,2){5}}
\put(55,5){\line(1,2){5}}
\put(25,15){\line(-3,-2){15}}
\put(60,15){\line(-3,-2){15}}
\put(10,5){\line(-1,2){5}}
\put(45,5){\line(-1,2){5}}
\put(5,15){\line(3,-2){15}}
\put(40,15){\line(3,-2){15}}
\put(15,15){\line(-1,-2){5}}
\put(50,15){\line(-1,-2){5}}
\put(80,15){\line(-1,-2){5}}
\put(75,5){\line(3,2){15}}
\put(90,15){\line(-1,-2){5}}
\put(85,5){\line(-1,2){5}}
\put(80,15){\line(3,-2){15}}
\put(95,5){\line(-1,2){5}}
\put(60,15){\line(3,-2){15}}
\put(120,15){\line(-1,0){10}}
\put(110,15){\line(-3,-2){15}}
\put(120,5){\line(1,1){10}}
\put(130,15){\line(0,-1){10}}
\put(130,5){\line(1,1){10}}
\put(140,15){\line(0,-1){10}}
\put(140,5){\line(-1,1){10}}
\put(130,5){\line(-1,1){10}}
\put(120,15){\line(2,-1){20}}
\put(120,5){\line(2,1){20}}
\put(-.07,9.93){\line(1,0){.9934}}
\put(1.916,9.93){\line(1,0){.9934}}
\put(3.903,9.93){\line(1,0){.9934}}
\put(5.89,9.93){\line(1,0){.9934}}
\put(7.877,9.93){\line(1,0){.9934}}
\put(9.863,9.93){\line(1,0){.9934}}
\put(11.85,9.93){\line(1,0){.9934}}
\put(13.837,9.93){\line(1,0){.9934}}
\put(15.824,9.93){\line(1,0){.9934}}
\put(17.811,9.93){\line(1,0){.9934}}
\put(19.797,9.93){\line(1,0){.9934}}
\put(21.784,9.93){\line(1,0){.9934}}
\put(23.771,9.93){\line(1,0){.9934}}
\put(25.758,9.93){\line(1,0){.9934}}
\put(27.744,9.93){\line(1,0){.9934}}
\put(29.731,9.93){\line(1,0){.9934}}
\put(31.718,9.93){\line(1,0){.9934}}
\put(33.705,9.93){\line(1,0){.9934}}
\put(35.691,9.93){\line(1,0){.9934}}
\put(37.678,9.93){\line(1,0){.9934}}
\put(39.665,9.93){\line(1,0){.9934}}
\put(41.652,9.93){\line(1,0){.9934}}
\put(43.638,9.93){\line(1,0){.9934}}
\put(45.625,9.93){\line(1,0){.9934}}
\put(47.612,9.93){\line(1,0){.9934}}
\put(49.599,9.93){\line(1,0){.9934}}
\put(51.585,9.93){\line(1,0){.9934}}
\put(53.572,9.93){\line(1,0){.9934}}
\put(55.559,9.93){\line(1,0){.9934}}
\put(57.546,9.93){\line(1,0){.9934}}
\put(59.532,9.93){\line(1,0){.9934}}
\put(61.519,9.93){\line(1,0){.9934}}
\put(63.506,9.93){\line(1,0){.9934}}
\put(65.493,9.93){\line(1,0){.9934}}
\put(67.479,9.93){\line(1,0){.9934}}
\put(69.466,9.93){\line(1,0){.9934}}
\put(71.453,9.93){\line(1,0){.9934}}
\put(73.44,9.93){\line(1,0){.9934}}
\put(75.426,9.93){\line(1,0){.9934}}
\put(77.413,9.93){\line(1,0){.9934}}
\put(79.4,9.93){\line(1,0){.9934}}
\put(81.387,9.93){\line(1,0){.9934}}
\put(83.373,9.93){\line(1,0){.9934}}
\put(85.36,9.93){\line(1,0){.9934}}
\put(87.347,9.93){\line(1,0){.9934}}
\put(89.334,9.93){\line(1,0){.9934}}
\put(91.32,9.93){\line(1,0){.9934}}
\put(93.307,9.93){\line(1,0){.9934}}
\put(95.294,9.93){\line(1,0){.9934}}
\put(97.281,9.93){\line(1,0){.9934}}
\put(99.267,9.93){\line(1,0){.9934}}
\put(101.254,9.93){\line(1,0){.9934}}
\put(103.241,9.93){\line(1,0){.9934}}
\put(105.228,9.93){\line(1,0){.9934}}
\put(107.214,9.93){\line(1,0){.9934}}
\put(109.201,9.93){\line(1,0){.9934}}
\put(111.188,9.93){\line(1,0){.9934}}
\put(113.175,9.93){\line(1,0){.9934}}
\put(115.161,9.93){\line(1,0){.9934}}
\put(117.148,9.93){\line(1,0){.9934}}
\put(119.135,9.93){\line(1,0){.9934}}
\put(121.122,9.93){\line(1,0){.9934}}
\put(123.109,9.93){\line(1,0){.9934}}
\put(125.095,9.93){\line(1,0){.9934}}
\put(127.082,9.93){\line(1,0){.9934}}
\put(129.069,9.93){\line(1,0){.9934}}
\put(131.056,9.93){\line(1,0){.9934}}
\put(133.042,9.93){\line(1,0){.9934}}
\put(135.029,9.93){\line(1,0){.9934}}
\put(137.016,9.93){\line(1,0){.9934}}
\put(139.003,9.93){\line(1,0){.9934}}
\put(140.989,9.93){\line(1,0){.9934}}
\put(142.976,9.93){\line(1,0){.9934}}
\put(144.963,9.93){\line(1,0){.9934}}
\put(146.95,9.93){\line(1,0){.9934}}
\put(148.936,9.93){\line(1,0){.9934}}
\put(0,5){\makebox(0,0)[cc]{$I$}}
\put(0,15){\makebox(0,0)[cc]{$R$}}
\put(25,19){\makebox(0,0)[cc]{$z_1$}}
\put(110,15){\line(1,-1){10}}
\put(110,19){\makebox(0,0)[cc]{$z_3$}}
\put(60,19){\makebox(0,0)[cc]{$z_2$}}
\qbezier(85,5)(41,-6)(25,15)
\put(73,3){\framebox(24,14)[cc]{}}
\put(85,20){\makebox(0,0)[cc]{$V_0$}}
\end{picture}
\hfill$\mbox{}$
\caption{A cubic graph $G$ with $\alpha(G)=\mu(G)=10$.}\label{fig2}
\end{figure}
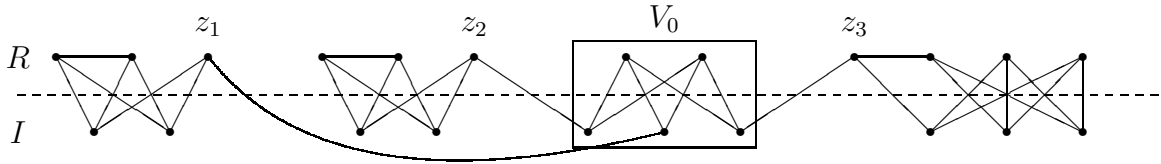

\begin{theorem}\label{theorem2}
A connected cubic graph $G$ satisfies $\alpha(G)=\mu(G)$ if and only if it is special.
\end{theorem}
\begin{proof}
First, we assume that $G$ is special.
For every $i$ in $\{ 0\}\cup [\ell]$, let $G_i=G[V_i]$.
Let the partite sets $I_0$ and $R_0$ of $G_0$ be as above.
For every $i$ in $[\ell]$, let the bubble $G_i$ 
have partition $(I_i,R_i)$.
Let $I=I_0\cup I_1\cup\ldots \cup I_\ell$
and $R=R_0\cup R_1\cup\ldots \cup R_\ell$.
Since $\alpha(G_i)=\alpha(G_i-z_i)$ for every $i$ in $[\ell]$,
we have 
$\alpha(G)=\alpha(G_0)+\sum_{i=1}^\ell\alpha(G_i)$.
By the vertex degrees, we have $|N_{G_0}(S)|\geq |S|$ for every set $S\subseteq R_0$, which implies $\alpha(G_0)=|I_0|$.
Together with Lemma \ref{lemma1}, we obtain
$\alpha(G)=|I_0|+\sum_{i=1}^\ell(n(G_i)-1)/2$.
Now, let $M$ be a maximum matching in $G$.
By Lemma \ref{lemma1},
the set $M$ contains at most $(n(G_i)-1)/2$
edges of $G_i$ for every $i$ in $[\ell]$,
which implies 
$\mu(G)\leq 
\mu(G')+\sum_{i=1}^\ell(n(G_i)-1)/2$,
where $G'=G[V_0\cup \{ z_1,\ldots,z_\ell\}]$.
Since $G'$ is bipartite with partite sets $I_0$ 
and $R_0\cup \{ z_1,\ldots,z_\ell\}$,
we obtain
$$
|I_0|+\sum_{i=1}^\ell\frac{n(G_i)-1}{2}
=\alpha(G)
\stackrel{(\ref{e1})}{\leq} \mu(G)
\leq |I_0|+\sum_{i=1}^\ell\frac{n(G_i)-1}{2},$$
in particular,
$\alpha(G)=\mu(G)$.

Now, let $G$ be a connected cubic graph with $\alpha(G)=\mu(G)$. 
Let $I$ be a maximum independent set in $G$, 
and let $R=V(G)\setminus I$. 
Let $V_1,\ldots,V_\ell$ be a maximal collection 
of disjoint sets of vertices of $G$ such that,
for every $i$ in $[\ell]$,
the graph $G[V_i]$ is a $2$-connected bubble
with contact vertex $z_i$ and partition $(I_i,R_i)$,
where $I_i\subseteq I$ and $R_i\subseteq R$, and 
the unique neighbor of $z_i$ outside of $V_i$ belongs to $I$.
Let $I_0=I\setminus (I_1\cup \ldots \cup I_\ell)$
and $R_0=R\setminus (R_1\cup \ldots \cup R_\ell)$.
If $R_0$ is an independent set, then $G$ is special.
Therefore, for a contradiction, 
we may assume that $xy$ is an edge of $G$ with $x,y\in R_0$.
Let $G'=G-\bigcup_{i\in [\ell]}(V_i\setminus \{ z_i\})$,
and let $G''=G'-\{ x,y\}$.
See Figure \ref{fig7} for an illustration.

\begin{figure}[H]
$\mbox{}$\hfill
\unitlength 1.5mm 
\linethickness{0.4pt}
\ifx\plotpoint\undefined\newsavebox{\plotpoint}\fi 
\begin{picture}(42,20)(0,0)
\put(25,5){\framebox(16,4)[cc]{$I'$}}
\put(17,4){\framebox(25,6)[lc]{$\,\,I_0$}}
\put(17,14){\framebox(25,6)[lc]{$\,\,R_0$}}
\put(23,15){\framebox(18,4)[cc]{$\,\,\,\,\,\,R'$}}
\put(25,16){\circle*{0.8}}
\put(29,16){\circle*{0.8}}
\put(25,16){\line(1,0){4}}
\put(25,17.5){\makebox(0,0)[cc]{\small $x$}}
\put(29,17.5){\makebox(0,0)[cc]{\small $y$}}
\put(-.07,11.93){\line(1,0){.9767}}
\put(1.883,11.93){\line(1,0){.9767}}
\put(3.837,11.93){\line(1,0){.9767}}
\put(5.79,11.93){\line(1,0){.9767}}
\put(7.744,11.93){\line(1,0){.9767}}
\put(9.697,11.93){\line(1,0){.9767}}
\put(11.651,11.93){\line(1,0){.9767}}
\put(13.604,11.93){\line(1,0){.9767}}
\put(15.558,11.93){\line(1,0){.9767}}
\put(17.511,11.93){\line(1,0){.9767}}
\put(19.465,11.93){\line(1,0){.9767}}
\put(21.418,11.93){\line(1,0){.9767}}
\put(23.372,11.93){\line(1,0){.9767}}
\put(25.325,11.93){\line(1,0){.9767}}
\put(27.279,11.93){\line(1,0){.9767}}
\put(29.232,11.93){\line(1,0){.9767}}
\put(31.186,11.93){\line(1,0){.9767}}
\put(33.139,11.93){\line(1,0){.9767}}
\put(35.092,11.93){\line(1,0){.9767}}
\put(37.046,11.93){\line(1,0){.9767}}
\put(38.999,11.93){\line(1,0){.9767}}
\put(40.953,11.93){\line(1,0){.9767}}
\put(0,7){\makebox(0,0)[cc]{$I$}}
\put(0,17){\makebox(0,0)[cc]{$R$}}
\put(13,17){\circle*{0.8}}
\put(4,17){\circle*{0.8}}
\multiput(4,17)(.054621849,-.033613445){238}{\line(1,0){.054621849}}
\multiput(13,17)(.033557047,-.046979866){149}{\line(0,-1){.046979866}}
\put(4,19){\makebox(0,0)[cc]{$z_1$}}
\put(13,19){\makebox(0,0)[cc]{$z_\ell$}}
\put(9,17){\makebox(0,0)[cc]{$\cdots$}}
\end{picture}
\hfill$\mbox{}$
\caption{The graph $G'$; 
removing from $G'$ the vertices $x$ and $y$ yields $G''$ while
removing from $G'$ all vertices from the bubble $B$ 
except for its contact vertex yields $G'''$.}\label{fig7}
\end{figure}
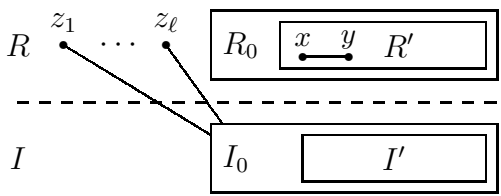
If $G''$ contains a matching $M_0$ that saturates $I_0$,
then, by Lemma \ref{lemma1}, 
the union of $\{ xy\}$,
the matching $M_0$, and 
maximum matchings in each $G[V_i]-z_i$
yields a matching of size more than $|I|=\alpha(G)=\mu(G)$,
which is a contradiction.
Hence, by Hall's theorem,
there is a set $I'\subseteq I_0$
with $|N_{G''}(I')|<|I'|$.
Let $R'=N_{G'}(I')$, and let $B=G[I'\cup R']$.
By the vertex degrees, the independence of $I$, and the edge $xy$,
we obtain $|R'|>|I'|$,
and, hence, $R'=N_{G''}(I')\cup \{ x,y\}$ and $|R'|=|I'|+1$,
see Figure \ref{fig7}.
Since $\sum_{u\in I'}d_{G'}(u)$ is a multiple of $3$,
this implies that $xy$ is the only edge of $G$ within $R'$,
that exactly one vertex in $R'$ has degree $2$ in $B$,
and that all remaining vertices of $B$ have degree $3$ in $B$,
that is, the graph $B$ is a bubble 
with some contact vertex $z$ and partition $(I',R')$.
Since each $z_i$ has degree $1$ in $G'$,
and $B$ contains no vertex of degree less than $2$,
we have $I'\subseteq I_0$ and $R'\subseteq R_0$,
see Figure \ref{fig7}.
By Lemma \ref{lemma1},
we may assume that $B$ is $2$-connected. 
Let $G'''=G'-(V(B)\setminus \{ z\})$.
Let $z'$ be the neighbor of $z$ outside of $V(B)$.
Suppose, for a contradiction, that $z'$ lies in $R_0$.
By the vertex degrees and the independence of $I$,
for every $S\subseteq I_0\setminus I'$,
we have $|N_{G'''}(S)|\geq |S|$,
and, in view of the edge $zz'$,
if $N_{G'''}(S)$ contains $z'$, then $|N_{G'''}(S)|>|S|$,
which implies $|N_{G'''-z'}(S)|\geq |S|$.
By Hall's theorem, 
the graph $G'''-z'$ has a matching saturating $I_0$,
which, by Lemma \ref{lemma1}, 
together with the edge $zz'$,
maximum matchings in each $G[V_i]-z_i$,
and a maximum matching in $B-z$
yields a matching of size more than $|I|=\alpha(G)=\mu(G)$,
which is a contradiction.
Hence, $z'\in I$. 
Now, the set $V(B)$ can be added 
to the collection $V_1,\ldots,V_\ell$,
contradicting its maximality,
which completes the proof.
\end{proof}
For a given connected cubic graph $G$,
the constructive proofs of Lemma \ref{lemma1}
and Theorem \ref{theorem2}
easily allow to design an efficient algorithm
that decides whether $\alpha(G)=\mu(G)$,
and, that returns the partition of $V(G)$
into the sets $V_0,V_1,\ldots,V_\ell$ in this case.
It remains an open problem to characterize the extremal $k$-regular graphs
for every $k$ at least $4$;
it might even be the case that these graphs are NP-hard to recognize.\\[3mm]
{\bf Acknowledgment} We thank
Yair Caro, Randy Davila, and Ryan Pepper
for valuable discussion and for sharing their many crucial examples
of extremal graphs.

\end{document}